\documentclass[reqno, 12pt]{amsart}

\usepackage[utf8]{inputenc}
\usepackage[T1]{fontenc}
\usepackage{lmodern}
\usepackage{amsmath}
\usepackage{amsthm}
\usepackage{amssymb}
\usepackage{latexsym}
\usepackage{amsmath,amsfonts,amssymb}
\usepackage{upgreek}
\usepackage{enumerate}
\usepackage{mathrsfs}
\usepackage{mathscinet}
\usepackage{stmaryrd} 
\usepackage{graphicx}
\usepackage{xcolor}
\usepackage[all,cmtip]{xy}
\usepackage{url}
\usepackage{hyperref}
\usepackage[a4paper, hmargin=3cm, vmargin=3cm]{geometry}

\newcommand\mathens[1]{\mathbb{#1}} 

\newcommand{\ud}{\mathrm{d}}

\newcommand{\N}{\mathens{N}}

\newcommand{\R}{\mathens{R}}
\newcommand{\C}{\mathens{C}}

\newcommand{\CP}{\C\mathrm{P}}

\newcommand{\RP}{\R\mathrm{P}}

\newcommand\sphere[1]{\mathens{S}^{#1}}

\newcommand{\ham}{\mathrm{Ham}}

\newcommand{\conto}{\mathrm{Cont}_{0}}

\newcommand{\id}{\mathrm{id}}
\newcommand{\nlmas}\upmu

\newtheorem{thm}{Theorem}[section]
\newtheorem{lem}[thm]{Lemma}
\newtheorem{cor}[thm]{Corollary}
\newtheorem{prop}[thm]{Proposition}

\newtheorem{questions}[thm]{Questions}
\newtheorem{prop-def}[thm]{Definition-proposition}

\theoremstyle{definition}
\newtheorem{definition}[thm]{Definition}

\theoremstyle{remark}

\newtheorem{rem}[thm]{Remark}

\newcommand\tpsi{\widetilde{\psi}}
\newcommand\tphi{\widetilde{\varphi}}

\newcommand\Leg{\mathcal{L}}
\newcommand\uLeg{\widetilde{\Leg}}
\newcommand\Gcont{\mathcal{G}}
\newcommand\uGcont{\widetilde{\Gcont}}
\newcommand\cleq{\preceq}
\newcommand\cgeq{\succeq}
\newcommand\spec{\mathrm{Spec}}

\newcommand\dSCH{\mathrm{d}_\mathrm{SCH}}
\newcommand\dSH{\mathrm{d}_\mathrm{SH}}
\newcommand\dspec{\mathrm{d}_\mathrm{spec}}

\newcommand\Nspec[1]{\left|#1\right|_\mathrm{spec}}

\newcommand\NSH[1]{\left|#1\right|_\mathrm{SH}}

\newcommand\LSCH{\mathrm{L}_\mathrm{SCH}}
\newcommand\LSH{\mathrm{L}_\mathrm{SH}}

\DeclareFontFamily{U}{mathb}{\hyphenchar\font45}
\DeclareFontShape{U}{mathb}{m}{n}{
      <5> <6> <7> <8> <9> <10> gen * mathb
      <10.95> mathb10 <12> <14.4> <17.28> <20.74> <24.88> mathb12
}{}
\DeclareSymbolFont{mathb}{U}{mathb}{m}{n}
\DeclareMathSymbol{\cll}{3}{mathb}{"CE}

\makeatletter\let\@wraptoccontribs\wraptoccontribs\makeatother

\begin{document}

\title{$C^1$-Local Flatness and Geodesics of the Legendrian Spectral Distance}
\author[S. Allais]{Simon Allais}
\address{S. Allais, Universit\'e de Strasbourg, IRMA UMR 7501, F-67000 Strasbourg, France}
\email{simon.allais@math.unistra.fr}
\urladdr{https://irma.math.unistra.fr/~allais/}

\author[P.-A. Arlove]{Pierre-Alexandre Arlove}
\address{P.-A. Arlove, Ruhr-Universit\"at Bochum, Fakult\"at f\"ur Mathematik, 
44780 Bochum, Germany}
\email{pierre-alexandre.arlove@ruhr-uni-bochum.de}
\date{march 15, 2024}

\subjclass[2020]{53D10, 57R17, 58B20}
\keywords{Orderability in contact geometry, spectral invariants, 
    Legendrian isotopies, contactomorphisms, Hofer's geometry,
geodesics}

\maketitle

\begin{abstract}
    In this article, we give an explicit computation of the order
    spectral selectors of a pair of $C^1$-close Legendrian submanifolds
    belonging to an orderable isotopy class. The $C^1$-local flatness of the
    spectral distance and the characterisation of its geodesics are
    deduced. Another consequence is the $C^1$-local coincidence of
    spectral and Shelukhin-Chekanov-Hofer distances. Similar
    statements are then deduced for several contactomorphism groups.
\end{abstract}

\section{Introduction}
Recently Nakamura \cite{Nakamura2023} and the authors of this article
\cite{allais2023spectral} defined independently the same distance on the
isotopy class of a closed Legendrian whenever it is orderable. While Nakamura
showed that the topology induced by this distance is the interval topology (see
\cite{CheNem2020} for a definition), the authors showed the spectrality of this
distance and therefore named it the spectral distance. Both results suggest the
natural character and importance of the spectral distance: on the one hand it
is a powerful object to study the geometry of this infinite dimensional space,
and on the other it allows to study and quantify contact dynamic
phenomena. We refer directly to \cite{Nakamura2023,allais2023spectral} for the
illustration of our previous words.

In this article we study the orderable isotopy class of a closed Legendrian 
submanifold endowed
with the spectral distance as a metric space on its own, being much inspired
by the seminal work of Bialy-Polterovich on the geodesics of Hofer's metric
\cite{BiaPol1994}.
In particular we show
that it is $C^1$-locally flat in the sense that $C^1$-locally it is isometric
to a normed vector space (see Section~\ref{sec:intro:flatness}). As a consequence we also get in this setting
the $C^1$-local flatness of the Shelukhin-Chekanov-Hofer distance
\cite{RosenZhang2020}. The flatness property allows us moreover
to give a complete description of the geodesics of these two distances (see
Section~\ref{sec:intro:geodesic}). Finally this allows us to get similar statements in some cases
for universal covers of certain contactomorphism groups.

\subsection{$C^1$-flatness}
\label{sec:intro:flatness}

From now on $(M,\xi=\ker\alpha)$ denotes a cooriented contact manifold
endowed with a contact form $\alpha$ the Reeb flow of which is complete.
We denote $\Leg$ (resp. $\uLeg$) the Legendrian isotopy class of some
closed Legendrian submanifold of $M$
(resp. the universal cover of $\Leg$) endowed with the $C^1$-topology.
Assuming $\Leg$ (resp. $\uLeg$) is orderable, let
$\ell_\pm^\alpha:\Leg\times\Leg\to\R$ (resp. $\uLeg\times\uLeg\to\R$)
denote the order $\alpha$-spectral selectors so that the $\alpha$-spectral
distance between two submanifolds $\Lambda_0,\Lambda_1\in\Leg$ 
(resp. $\uLeg$) is defined as
\[
    \dspec^\alpha(\Lambda_0,\Lambda_1) := \max\left\{ \ell_+^\alpha(\Lambda_0,\Lambda_1),
    -\ell_-^\alpha(\Lambda_0,\Lambda_1) \right\}\in [0,+\infty).
\]
The map $\dspec^\alpha$ is a genuine distance on $\Leg$ and is \textit{a priori} just a pseudo-distance on $\uLeg$. We refer to Section~\ref{sec:spec} for details.

Let $\Lambda\in\Leg$ be a closed Legendrian submanifold.
According to the Weinstein neighborhood theorem,
there is an open set of $M$ containing $\Lambda$ which is contactomorphic
to an open set of $J^1\Lambda$ containing the zero-section, identifying $\Lambda$
with the zero-section and the contact form $\alpha$ with the canonical contact form
$\alpha_0:=\ud z - p\cdot\ud q$. Every Legendrian $C^1$-close to $\Lambda$ in $\Leg$ is then identified
uniquely with the 1-jet of a map $f\in C^\infty(\Lambda,\R)$ of some $C^2$-neighborhood
$U$ of the zero map.
Let us call the induced continuous embedding
$\Phi : U\to \Leg$ an $\alpha$-Weinstein parametrization of $\Leg$ centered at $\Lambda$
(it is a homeomorphism between $U$ and a $C^1$-neighborhood of $\Lambda$).
As $\Leg$ and $\uLeg$ are locally homeomorphic, one naturally extends the notion
of Weinstein parametrization to $\uLeg$.

Following the terminology introduced by Bialy-Polterovich in \cite{BiaPol1994},
we say that $\Leg$ (resp. $\uLeg$, and $\Gcont$ or $\uGcont$ introduced in Section \ref{ssec:conventions}) endowed with 
a pseudo-distance $d$ is $C^1$-locally flat if at any point of it there exists a $C^1$-neighborhood
on which $d$ is isometric to the restriction of a normed
distance to some open neighborhood of a vector space.
In our cases, the vector space will always be $C^\infty(\Lambda,\R)$ endowed with the $C^0$-norm
$f\mapsto \max|f|$ for some closed manifold $\Lambda$.

\begin{thm}\label{thm}
 If $\Leg$ (resp. $\uLeg$) is orderable then endowed with the Legendrian
 spectral distance it is $C^1$-locally flat. More precisely, for every
 $\Lambda\in\Leg$ (resp. $\uLeg$), and every $\alpha$-Weinstein parametrization
 $\Phi:U\to\Leg$ (resp. $U\to\uLeg$) centered at $\Lambda$, there exists
 $U'\subset U$ a $C^2$-neighborhood of the zero map such that for all $f,g\in U'$
   \[\ell_+^\alpha(\Phi(f),\Phi(g))=\max (f-g)\quad \quad \text{ and }\quad \quad  \ell_-^\alpha(\Phi(f),\Phi(g))=\min (f-g),\]
 in particular 
   $\dspec^\alpha(\Phi(f),\Phi(g))=\max |f-g|$.
\end{thm}

Some known examples of orderable $\Leg$ or $\uLeg$ are the isotopy class of
the zero-section of any 1-jet space over a closed manifold, the universal cover of the
isotopy class of $\RP^n$
in the standard contact $\RP^{2n+1}$, the universal cover of the isotopy class of a fiber
$S^*_x X$ of any unit cotangent bundle (see \emph{e.g.}
\cite[Examples~2.10]{allais2023spectral} for references).

An analogous statement can be given for contactomorphisms when appropriate
spectral selectors exist. For instance such exist when $\uLeg(\Delta)$ is orderable
for the contact product $M\times M\times\R$
(see Section~\ref{sec:contacto}) which is the case when $M$ is a hypertight closed contact manifold or a closed unit cotangent bundle or any contact boundary
of a compact Liouville domain, the symplectic homology of which does not vanish
(we also refer to \cite[Examples~2.10]{allais2023spectral}). Other situations in which suitable spectral selectors exist concern universal covers of contactomorphism groups of contact lens spaces \cite{allais2024spectral}.

\begin{rem}\label{rem:1 jet}
    Note that for the $1$-jet bundle of a closed manifold $X$ a stronger statement
    can be directly deduced from Corollary 5.4 of \cite{chernov2011legendrian}.
    This Corollary tells us that $j^1f\cleq j^1g$ (\emph{cf.}
    Section~\ref{ssec:conventions}) if and only if $f\leq g$ everywhere, where
    $f,g :X\to \R$ are smooth functions and $j^1f$, $j^1g$ denote the graph of
    their respective $1$-jet. This indeed implies that
    $\ell_+^{\alpha_0}(j^1f,j^1g)=\max (f-g)$ and
    $\ell_-^{\alpha_0}(j^1f,j^1g)=\min (f-g)$, where $\alpha_0=\ud z-p\cdot\ud
    q$ is the canonical contact form of $J^1X=T^*X\times\R$.
\end{rem}

An important consequence of Theorem~\ref{thm} is that the
Shelukhin-Chekanov-Hofer (SCH) distance $\dSCH^\alpha$ agrees $C^1$-locally
with the spectral distance $\dspec^\alpha$, and therefore is $C^1$-locally flat
(we refer to Section~\ref{sec:SHdist} for definitions).

\begin{cor}\label{cor:shel}
        Suppose $\Leg$ (resp. $\uLeg$) is orderable. Then for every
        $\Lambda\in\Leg$ (resp. $\uLeg$) there exists $\mathcal{U}$ a
        $C^1$-neighborhood of $\Lambda$ such that 
        \[\dSCH^\alpha(\Lambda_1,\Lambda_0)=\dspec^\alpha(\Lambda_1,\Lambda_0)\text{ for all } \Lambda_1,\Lambda_0\in\mathcal{U}.\] Therefore endowed with the SCH distance $\Leg$ (resp. $\uLeg$) is $C^1$-locally flat. 
\end{cor}

Let us discuss another corollary which has motivated the writing of Theorem
\ref{thm}. Recall from \cite{Nemirovski_dejavu} that a Legendrian isotopy
$(\Lambda_t)\subset\Leg$ (resp. $\uLeg$) is said to be monotone if
$\Lambda_t\cleq\Lambda_s$ whenever $t\leq s$ (\emph{cf.}
Section~\ref{ssec:conventions}). The next corollary answers
\cite[Question 2.4.]{Nemirovski_dejavu}.

\begin{cor}
If $\Leg$ is orderable (resp. $\uLeg$ is orderable) then an isotopy
$(\Lambda_t)\subset\Leg$ (resp. $\uLeg$) is monotone if and only if it is
non-negative.  
\end{cor}

\begin{proof}
    Suppose by contradiction that a monotone isotopy $(\Lambda_t)$ is not
    non-negative at some time $t_0\in[0,1]$. One can assume $t_0=0$.
    Let $\Phi:U\to\Leg$ be a parametrization centered at $\Lambda_0$
    given by Theorem~\ref{thm}.
    Therefore for $t>0$ small
    enough $\Lambda_t$ lies in $\Phi(U)$ and its corresponding function
    $f_t:=\Phi^{-1}(\Lambda_t)\in U$ is such that $\min f_t<0$. Thus on the one
    hand $\ell_-^\alpha(\Lambda_t,\Lambda_0)<0$ by Theorem \ref{thm}. But on the
    other hand $\ell_-^\alpha(\Lambda_t,\Lambda_0)\geq 0$ since
    $\Lambda_0\cleq\Lambda_t$ which brings the contradiction.
\end{proof}

\subsection{Geodesics}\label{sec:intro:geodesic}

Note that Theorem \ref{thm} directly
implies that ``straight'' paths $t\mapsto\Phi(tf)$ are minimizing geodesics for the
spectral distance. More precisely recall that in a pseudo-metric space $(X,d)$ the
length induced by the pseudo-distance $d$ of a continuous curve $\gamma : [a,b]\to X$,
for some real numbers $a\leq b$ is defined as follows 
\begin{equation}\label{eq:Length}
\mathrm{Length}_d(\gamma):=\sup\left\{\sum\limits_{i=1}^k
d(\gamma(t_{i-1}),\gamma(t_i))\ |\ k\in\N\ , a=t_0<\cdots< t_k=b\right\}.
\end{equation}
In the following $I\subset\R$ will denote an interval and a path $\gamma :I\to X$ is by definition a continuous map. 
\begin{definition}
    A path $\gamma : I\to X$ is a \emph{minimizing geodesic} if
$d(\gamma(a),\gamma(b))=\mathrm{Length}_d(\gamma|_{[a,b]})$ for all $a,b\in I$ such that $a<b$.  A path $\gamma:
    I\to X$ is a \emph{geodesic} if for all $t\in I$ there exists a
    neighborhood $J\subset I$ of $t$ such that
    $\gamma|_{J}$ is a minimizing geodesic.
\end{definition}
  
  \begin{rem}\label{rem:longueur spectral et shelukhin}\
  \begin{enumerate}
      \item In our situation we are interested only in the subset of paths
          consisting of smooth isotopies. Corollary \ref{cor:shel} implies that
          the SCH-length functional and spectral length functional agree on
          smooth isotopies (see Section~\ref{sec:SHdist}).
  \item In our situation it should also be possible to define geodesics as critical points of
      the length functional. Indeed, even if it is not clear that the length functional
      is smooth in general, it should be the case at paths satisfying
      the previous definition (see \cite[Chapter 12]{polterovich} and \cite[Section 1.2]{geometricvariants}).
  \end{enumerate}
  \end{rem}
  The main result of this section consists of giving a complete
  characterization of smooth geodesics of the spectral and SCH
  distances. To do so we introduce the following notion
  extending the terminology introduced by Bialy-Polterovich
  \cite[Definition 1.3.C]{BiaPol1994} (see also Section \ref{ssec:conventions}
  for the definition of Hamiltonian maps).
  
  \begin{definition}\label{def:quasiautonome}
    A Legendrian isotopy $(\Lambda_t)_{t\in I}$ is
    \emph{$\alpha$-quasi-autonomous} if there exist $\epsilon\in\{\pm1\}$ and a continuous path
    $(x_t)$ of points belonging to a same $\alpha$-Reeb orbit such that
    $x_t\in \Lambda_t$ and $\epsilon H_t(x_t) = \max | H_t|$, $\forall t\in I$ where
    $(H_t)$ denotes the $\alpha$-Hamiltonian map of $(\Lambda_t)$.\\
    A Legendrian isotopy $(\Lambda_t)_{t\in I}$ is
    \emph{locally} $\alpha$-quasi-autonomous if for all $t\in I$
    there exists a neighborhood $J\subset I$ of $t$
    such that $(\Lambda_t)_{t\in J}$ is $\alpha$-quasi-autonomous.
\end{definition}

\begin{thm}\label{thm:geo}
  A Legendrian isotopy $(\Lambda_t)$ in an orderable $\Leg$ (resp. $\uLeg$) is
  a geodesic for $\ud^\alpha$ if and only if it is $\alpha$-quasi-autonomous,
  where $\ud^\alpha$ denotes either $\dspec^\alpha$ or $\dSCH^\alpha$. 
\end{thm}

The above Theorem \ref{thm:geo} and Definition \ref{def:quasiautonome} have
their direct analogue for contact isotopies (see Section~\ref{sec:contacto}).

\subsection{Discussion}

In the symplectic context similar statements have been proven for the Hofer
norm and for the Viterbo's type spectral norms \cite{Ohnorm,schwarz,Vit} on the
group of Hamiltonian symplectomorphisms
\cite{BiaPol1994,lalondemcduff12,geometricvariants,Ohplat}. Bialy-Polterovich \cite{BiaPol1994} characterized the geodesics of the Hofer norm of
compactly supported Hamiltonian symplectomorphisms of the standard symplectic
Euclidean space after proving the local flatness. Lalonde-McDuff
\cite{lalondemcduff12} went the other way around by first characterizing
geodesics of symplectic manifolds that do not admit short loops, i.e. $0$ is an
isolated point of the Hofer length spectrum, and derive from it local flatness
for this family of symplectic manifolds. Finally McDuff \cite{mcduffsalamon}
was able to get rid of the assumption about the non-existence of short loops and
thus characterized in full generality the geodesics and proved the local
flatness of the Hofer norm for all closed symplectic manifolds. To do so, she
generalized the non-squeezing theorem of Gromov \cite{Gro85} to non-trivial
symplectic fibrations over the $2$-sphere. Her proof involves technical tools such
as Seidel morphism and Gromov-Witten invariants.

Surprisingly enough our proofs of local flatness and characterization of
geodesics in the contact context involve only elementary arguments relying
essentially on the axiomatic properties of the Legendrian spectral distance
that we list below (see Theorem \ref{thm:LegSS}). As in
\cite{allais2023spectral}, the explanation of this surprising ease comes from
the orderability assumption. Indeed, non trivial machineries, such as Floer
Homology or generating functions techniques, are hidden behind this assumption.
Nevertheless once this assumption has been made, the contact spectral distance
is well defined and seems easier and more intuitive to handle than its
symplectic cousins.

As illustrated by the present article, it is very common that one adapts
statements from symplectic geometry to statements in contact geometry. However
according to the previous paragraph it should also be interesting in the future
to go the other way around and see whether one can use the contact spectral
distance to shed some light on the Hofer or Viterbo's type distances of
symplectic geometry.

\subsection*{Organization of the article}
In Section~\ref{sec:prelim}, we fix the notations and the conventions on the
objects that we will use throughout the paper.
In Section~\ref{sec:proofMain}, we prove Theorem~\ref{thm} and
Corollary~\ref{cor:shel}. In Section~\ref{sec:preuvegeo},
we prove Theorem~\ref{thm:geo} characterizing geodesics of the
spectral distance.
Finally, in Section~\ref{sec:contacto}, the extension to universal covers of some
contactomorphism groups is discussed.

\subsection*{Acknowledgment} 
The second author wants to thank Sheila Sandon for introducing him to symplectic geometry with the paper of Bialy-Polterovich \cite{BiaPol1994}. We both want to thank Alberto Abbondandolo, Lukas Nakamura, Stefan Nemirovski and Sheila Sandon for their questions and remarks that have motivated and enlightened this work.\\
The first author was funded by the postdoctoral
fellowships of the \emph{Interdisciplinary Thematic Institute IRMIA++}, as part
of the \emph{ITI} 2021-2028 program of the University of Strasbourg, CNRS and
Inserm, supported by \emph{IdEx Unistra (ANR-10-IDEX-0002)} and by
\emph{SFRI-STRAT’US project (ANR-20-SFRI-0012)} under the framework of the
\emph{French Investments for the Future Program}.  The second author is
partially supported by the \emph{Deutsche Forschungsgemeinschaft} under the
\emph{Collaborative Research Center SFB/TRR 191 - 281071066 (Symplectic
Structures in Geometry, Algebra and Dynamics)}.

\section{Preliminaries}\label{sec:prelim}

\subsection{Conventions}\label{ssec:conventions}

Let $(M,\xi)$ be a cooriented contact manifold and fix a contact form $\alpha$
supporting $\xi$ and its coorientation whose Reeb vector field is complete and
denote by $\phi_t^\alpha\in\conto(M,\xi)$ its flow at time $t\in\R$. Here
$\conto(M,\xi)$ denotes the group of contactomorphisms isotopic to the
identity. We denote by $\Gcont$ the
group of contactomorphisms isotopic to the identity through compactly supported
contactomorphisms and we endow it with the $C^1$-topology. We denote by
$\uGcont$ its universal cover and $\Pi :\uGcont\to\Gcont$ the covering map. 
By a slight abuse of notation, we still call the identity and
denote $\id\in\uGcont$ the class of the constant
isotopy $s\mapsto \id$.

Fix a closed Legendrian $\Lambda\subset M$ and denote by
$\Leg(\Lambda)=\{\phi(\Lambda)\ |\ \phi\in\conto(M,\xi)\}$ (or simply $\Leg$)
its isotopy class that we endow with the $C^1$-topology. We denote by $\uLeg$
its universal cover and $\Pi :\uLeg\to\Leg$ the covering map.

Everywhere in the paper, $I\subset\R$ will denote an interval.
Let us recall that a Legendrian isotopy $(\Lambda_t)_{t\in I}$
in $\Leg$ is a path of Legendrian submanifolds such that there exists
a smooth map $j:I\times \Lambda\to M$ whose restriction $j_t$ to $\{ t\}\times \Lambda$
is an embedding onto $\Lambda_t$ for all $t\in I$.
The $\alpha$-Hamiltonian map of $(\Lambda_t)$ is the family of maps
$(h_t : \Lambda_t\to\R)$ defined by $h_t\circ j_t = \alpha(\partial_t j_t)$
for any smooth parametrization $j$.
Given a (smooth) contact isotopy $(\phi_t)\subset \conto(M,\xi)$, we recall that
its $\alpha$-Hamiltonian map $h:I\times M\to\R$ is defined by
$h_t\circ\phi_t = \alpha(\partial_t \phi_t)$.
 We say that a path
$(\Lambda_t)\subset\uLeg$ (resp. $(\phi_t)\subset\uGcont$) is an isotopy if its
projection to $\Leg$ (resp. $\Gcont$) is an isotopy and we define its
$\alpha$-Hamiltonian to be the $\alpha$-Hamiltonian of its projection. 
  
On $O$ being either $\Gcont,\uGcont,\Leg$ or $\uLeg$ we write $x\cleq y$, or
equivalently $y\cgeq x$, if there exists a non-negative isotopy from $x$ to
$y$, \emph{i.e.} an isotopy whose Hamiltonian is non-negative. $O$ is called
orderable if and only if $\cleq$ defines a partial order. 
This relation has been introduced by Eliashberg-Polterovich \cite{EP00} and has
been widely studied
\cite{allais2023spectral,bhupal,chernovnemirovski2,CFP,EKP,Sandon2010}.

\subsection{Hofer type distances}
\label{sec:SHdist}

The Shelukhin-Chekanov-Hofer (resp. Shelukhin-Hofer) length functional is a
functional defined on the space of Legendrian isotopies $(\Lambda_t)\subset\Leg$
(resp. contact isotopies $(\phi_t)\subset\Gcont$) as follows: the
length of any isotopy of $\alpha$-Hamiltonian $(H_t)_{t\in I}$ is given by 
\begin{equation*}
\int_I\max |H_t|dt,
\end{equation*}
(see also \cite[Section 7]{RosenZhang2020} and \cite{shelukhin}).
We denote this length functional by $\LSCH^\alpha$ (resp. $\LSH^\alpha$).
The Shelukhin-Chekanov-Hofer (SCH) pseudo-distance $\dSCH^\alpha$ on $\Leg$ or $\uLeg$ 
is defined for any $\Lambda_0,\Lambda_1\in\Leg$ or $\uLeg$
as
\[
    \dSCH^\alpha(\Lambda_1,\Lambda_0):=\inf\{\LSCH^\alpha(\Lambda_t)\ |\
    (\Lambda_t)\subset\Leg \text{ or }\uLeg \text{ joining $\Lambda_0$ to
$\Lambda_1$}\}.
\]
The Shelukhin-Hofer pseudo-norm $\NSH{\cdot}^\alpha$ on $\Gcont$ or $\uGcont$
is defined for any $\phi\in\Gcont$ or $\uGcont$ as
\[
    \NSH{\phi}^\alpha=\inf\{\LSH^\alpha(\phi_t)\ |\ (\phi_t)_{t\in [0,1]}\subset \Gcont \text{
or } \uGcont \text{ such that } \phi_0=\id \text{ and
}\phi_1=\phi\}.
\]

When $M$ is closed Shelukhin \cite{shelukhin} showed that $\NSH{\cdot}^\alpha$
is a genuine norm on $\Gcont$. Hedicke \cite{Hedicke2022} showed that
$\dSCH^\alpha$ is a genuine distance when $\Leg$ is orderable (even for
non-closed $M$).

\begin{rem}\label{rem:longueur de shelukhin}
    Note that Theorem \ref{thm} together with Corollary \ref{cor:shel} imply
    that $\mathrm{Length}_{\dSCH^\alpha}$ and $\mathrm{Length}_{\dSH^\alpha}$
    (as defined in (\ref{eq:Length})) restricted to isotopies correspond
    respectively to $\LSCH^\alpha$ and $\LSH^\alpha$, where $\dSH^\alpha$
    denotes the right-invariant distance associated with the norm
    $\NSH{\cdot}^\alpha$. See also  \cite[Proposition 1.6]{gromovmetrique} and
    the remark following it. 
\end{rem}

\subsection{Order spectral selectors and induced distances}
\label{sec:spec}

Following \cite{allais2023spectral} let us define
two functions $\ell_\pm^\alpha:\Leg\times\Leg\to\R\cup\{\mp\infty\}$ (resp.
$\uLeg\times\uLeg\to\R\cup\{\mp\infty\}$) by
\[
\ell_+^\alpha(\Lambda_1,\Lambda_0):=\inf\{t\in\R\ |\
\Lambda_1\cleq\phi_t^\alpha\cdot\Lambda_0\} \text{ and }
\ell_-^\alpha(\Lambda_1,\Lambda_0):=\sup\{t\in\R\ |\
\phi_t^\alpha\cdot\Lambda_0\cleq\Lambda_1\},
\]
for $\Lambda_1,\Lambda_0\in\Leg$ (resp. $\Lambda_1,\Lambda_0\in\uLeg$), where
$\phi_t^\alpha\cdot\Lambda_0$ denotes the natural action of the Reeb flow at
time $t\in\R$ on $\Lambda_0\in\Leg$ (resp. $\Lambda_0\in\uLeg$). Recall that the $\alpha$-spectrum of $(\Lambda_1,\Lambda_0)\in\Leg^2$ (resp. $\uLeg^2$) is the set of lengths of $\alpha$-Reeb chords joining $\Lambda_0$ to $\Lambda_1$ (resp. $\Pi(\Lambda_0)$ to $\Pi(\Lambda_1)$) that is
\[\spec^\alpha(\Lambda_1,\Lambda_0):=\{t\in\R\ | \ \Lambda_1\cap \phi_t^\alpha\Lambda_0\ne\emptyset\} \text{ (resp. } \spec^\alpha(\Lambda_1,\Lambda_0):=\spec^\alpha(\Pi\Lambda_1,\Pi\Lambda_0)\text{)}.\]

\begin{thm}[\cite{allais2023spectral}]\label{thm:LegSS}
    The maps $\ell_\pm^\alpha$ are real-valued if and only if $\Leg$ (resp. $\uLeg$) is orderable. Moreover when real valued they satisfy the
    following properties for every $\Lambda_0,\Lambda_1,\Lambda_2\in\Leg$ (resp. $\uLeg$),
    \begin{enumerate}[ 1.]
        \item (normalization) $\ell_\pm^\alpha(\Lambda_0,\Lambda_0) = 0$
            and $\ell_\pm^\alpha(\phi^\alpha_t \Lambda_1,\Lambda_0)
            = t + \ell_\pm^\alpha(\Lambda_1,\Lambda_0)$, $\forall t\in\R$,
        \item (monotonicity) $\Lambda_2\cleq \Lambda_1$ implies
            $\ell_\pm^\alpha(\Lambda_2,\Lambda_0) \leq \ell_\pm^\alpha(\Lambda_1,\Lambda_0)$,
        \item (triangle inequalities) $\ell_+^\alpha(\Lambda_2,\Lambda_0) \leq
            \ell_+^\alpha(\Lambda_2,\Lambda_1) + \ell_+^\alpha(\Lambda_1,\Lambda_0)$
            and $\ell_-^\alpha(\Lambda_2,\Lambda_0) \geq \ell_-^\alpha(\Lambda_2,\Lambda_1)
            + \ell_-^\alpha(\Lambda_1,\Lambda_0)$,
        \item (Poincaré duality) $\ell_+^\alpha(\Lambda_1,\Lambda_0) =
            - \ell_-^\alpha(\Lambda_0,\Lambda_1)$,
        \item (compatibility) $\ell_\pm^\alpha(\varphi(\Lambda_1),\varphi(\Lambda_0))
            = \ell_\pm^{\varphi^*\alpha}(\Lambda_1,\Lambda_0)$,
            for every $\varphi$ in $\conto(M,\xi)$
            (resp. in its universal cover),
        \item (non-degeneracy) $\ell_+^\alpha(\Lambda_1,\Lambda_0) = \ell_-^\alpha(\Lambda_1,\Lambda_0) = t$
            for some $t\in\R$ implies $\Lambda_1 = \phi_t^\alpha \Lambda_0$
            (resp. it only implies the equality
            $\Pi \Lambda_1 = \phi_t^\alpha\Pi \Lambda_0$ in $\Leg$).
        \item (spectrality) $\ell_\pm^\alpha(\Lambda_1,\Lambda_0) \in
            \spec^\alpha(\Lambda_1,\Lambda_0)$.
    \end{enumerate}
\end{thm}

As a consequence the map $\dspec^\alpha:=\max\{\ell_+^\alpha,-\ell_-^\alpha\}$
is a distance (resp. a pseudo-distance) on $\Leg$ (resp. $\uLeg$) whenever it
is orderable. Thanks to the last property of Theorem \ref{thm:LegSS} we call
this (pseudo-)distance the Legendrian spectral distance. 
A consequence of normalization and monotonicity properties are the
following inequalities.
If $(\Lambda_t)_{t\in[0,1]}$ is an isotopy of $\Leg$ (resp. $\uLeg$), then
\begin{equation}\label{eq:LselectorsHam}
    \int_0^1 \min H_t\ud t \leq \ell_-^\alpha(\Lambda_1, \Lambda_0)
    \leq \ell_+^\alpha(\Lambda_1, \Lambda_0)
    \leq\int_0^1 \max H_t\ud t,
\end{equation}
where $(H_t)$ is the associated $\alpha$-Hamiltonian map 
\cite[Lemma~3.3]{allais2023spectral}.
In particular, the spectral distance is dominated by the SCH distance:
\begin{equation}\label{eq:dspecdSCH}
    \dspec^\alpha \leq \dSCH^\alpha,
\end{equation}
which subsequently implies the $C^1$-continuity of $\ell_\pm^\alpha$
\cite[Corollary~3.4]{allais2023spectral}.

Assuming that $M$ is closed, similarly we defined in \cite{allais2023spectral,ArlovePhD}
two functions $c_\pm^\alpha : \Gcont\to\R\cup\{\mp\infty\}$ (resp.
$\uGcont\to\R\cup\{\mp\infty\}$) 
\[
    c_+^\alpha(\phi)=\inf\{t\in\R\ |\ \phi\cleq\phi_t^\alpha\}\quad
    \text{ and }
    \quad c_-^\alpha(\phi)=\sup\{t\in\R\ |\ \phi_t^\alpha\cleq\phi\}.
\]
We showed that the maps $c_\pm^\alpha$ take values in $\R$ if and only if
$\Gcont$ (resp. $\uGcont$) is orderable and moreover $c_\pm^\alpha$ satisfy
properties analogous to the ones of $\ell_\pm^\alpha$ except for the
spectrality. Let us recall that the $\alpha$-spectrum of $\phi\in\Gcont$ (resp.
$\phi\in\uGcont$) is the set of translations of its $\alpha$-translated points
that is
\begin{equation}\label{eq:Gspectrum}
    \spec^\alpha(\phi):=\{t\in\R\ |\ \exists x\in M,\ \phi(x)=\phi_t^\alpha(x),\
    (\phi^*\alpha)_x=\alpha_x\}
\end{equation}
\[\text{(resp. }\spec^\alpha(\phi):=\spec^\alpha(\Pi(\phi))\text{)}.\]
Nevertheless this allows us to define a norm (resp. a pseudo-norm)
\[
    \Nspec{\cdot}^\alpha:=\max\{c_+^\alpha,-c_-^\alpha\}
\]
that we still call the
spectral norm on $\Gcont$ (resp. $\uGcont$). 
Since normalization and monotonicity are still satisfied by $c_\pm^\alpha$, one
has analogues of (\ref{eq:LselectorsHam}) and (\ref{eq:dspecdSCH}).
Given $\phi\in\Gcont$ (resp. $\in\uGcont$),
\begin{equation}\label{eq:GselectorsHam}
    \int_0^1 \min H_t\ud t \leq c_-^\alpha(\phi)
    \leq c_+^\alpha(\phi)
    \leq\int_0^1 \max H_t\ud t,
\end{equation}
for any $\alpha$-Hamiltonian map $(H_t)$ generating $\phi$.
As a consequence $\Nspec{\cdot}^\alpha \leq \NSH{\cdot}^\alpha$.

\section{Proof of Theorem \ref{thm}}\label{sec:proofMain}

In order to prove Theorem~\ref{thm} let us assume $\Leg$ (resp. $\uLeg$)
to be orderable and let us fix once for all $\Lambda$ in it.
Let us be more explicit on the construction of the parametrization
$\Phi$ centered at $\Lambda$ in order to fix notation.
By the Weinstein neighborhood theorem, there exists $\Psi$ a diffeomorphism
between a neighborhood $\mathcal{V}\subset M$ of $\Lambda$ and a neighborhood
$V\subset J^1\Lambda$ of the $0$-section $j^10$ that moreover satisfies
$\Psi(\Lambda)=j^10$, and more precisely $\Psi(x)$ is the image of the $0$-section at $x$ for all $x\in\Lambda$, and $\Psi^*\alpha_0=\alpha$ where $\alpha_0=\ud z-p\cdot\ud q$ denotes the canonical $1$-form of $J^1\Lambda$. Let $U$ be a sufficiently 
$C^2$-small open neighborhood of the $0$-function in $C^\infty(\Lambda,\R)$
such that $j^1f\subset V$ for any $f\in U$. Then one can show that the map 
\[
\Phi : U\to \Leg  \quad \quad \quad f\mapsto \Psi^{-1}(j^1f)
\]
is injective and open (see for example the third paragraph of \cite{tsuboi3}
for more details).

One can have the same discussion for the universal cover $\uLeg$ of $\Leg$. Indeed in this situation, ensuring that $\Phi(U)$ is small enough, one can use the projection $\Pi : \uLeg\to\Leg$, which is a local homeomorphism, to construct the local homeomorphism $\Pi|_{\Phi(U)}^{-1}\circ\Phi : U\to\uLeg$ . By a slight abuse of notation we also denote this latter local homeomorphism by $\Phi$.

Before proving Theorem \ref{thm} let us prove the following lemmata.

\begin{lem}\label{lem1}
  Let $f\in U$. If an $\alpha$-Reeb chord between $\Lambda$ and $\Phi(f)$ of
  length $\ell\in\R$ is contained in $\mathcal{V}$, i.e.
  $\{\phi^\alpha_{t\ell}(x)\}_{t\in[0,1]}\subset\mathcal{V}$ for some
  $x\in\Lambda$ and $\phi_\ell^\alpha(x)\in\Phi(f)$, then $\ell$ is a critical
  value of $f$.
\end{lem}

\begin{proof}
Since $\Psi :\mathcal{V}\to V$ is a strict contactomorphism
$\Psi(\phi_{t\ell}^\alpha(x))=(x,0,t\ell)\in V\subset T^*X\times \R$. Therefore
$(x,0,\ell)\in j^1f=\{(x,\ud f(x),f(x))\ |\ x\in \Lambda\}$ which brings the
conclusion.
\end{proof}

Let $\varepsilon_0>0$ be sufficiently small so that:
\begin{enumerate}[(i)]
    \item\label{propriete1} $\phi_t^\alpha(\Lambda)\subset\mathcal{V}$ for any $t\in(-\varepsilon_0,\varepsilon_0)$
    \item there exists an open set of the form $V_{\varepsilon_0}:=\underline{V}\times (-\varepsilon_0,\varepsilon_0)\subset J^1\Lambda=T^*\Lambda\times\R$ that is contained in $V$.
\end{enumerate}

From now on, for any positive $\varepsilon\leq\varepsilon_0$ we will denote by
$U_\varepsilon$ a convex $C^2$-neighborhood of the $0$-function such that
$j^1f\in V_{\varepsilon}:=\underline{V}\times (-\varepsilon,\varepsilon)$ for
all $f\in U_\varepsilon$. 

\begin{cor}\label{cor:critval}
   Let $f\in U_{\varepsilon}$ then $\ell_\pm^\alpha(\Phi(f),\Lambda)$ are critical values of $f$. 
\end{cor}

\begin{proof}
    Since $f\in U_\varepsilon$ by (\ref{eq:LselectorsHam}) it implies that $-\varepsilon<\min
    f\leq\ell_-^\alpha(\Phi(f),\Lambda)\leq \ell_+^\alpha(\Phi(f),\Lambda)\leq
    \max f<\varepsilon$. Since $\ell_\pm:=\ell_\pm^\alpha(\Phi(f),\Lambda)$ are
    spectral values it implies that there exists $x_\pm\in\Lambda$ such that
    $\{\phi_\alpha^{t\ell_{\pm}}(x_\pm)\}$ are Reeb chords between $\Lambda$
    and $\Phi(f)$. By \eqref{propriete1} these Reeb chords are included in
    $\mathcal{V}$ and therefore we conclude by Lemma \ref{lem1}.
\end{proof}

\begin{lem}\label{lem0}
There exists a positive $\delta_0\leq\varepsilon_0$  such that for any positive $\delta\leq\delta_0$
\[\ell_\pm^\alpha(\Phi(f),\Phi(g))=\ell_\pm^\alpha(\Phi(f-g),\Lambda)\quad \text{for any }f,g\in U_\delta.\]

\end{lem}
\begin{proof}
Note that if there exists a contactomorphism $\phi\in\conto(M,\ker\alpha)$ that
commutes with the Reeb flow, or equivalently $\phi^*\alpha=\alpha$, such that 
\begin{equation}\label{eq:contstrict}
\phi(\Phi(h))=\Phi(h-g) \text{ for any } h\in U_\delta
\end{equation}
then Lemma \ref{lem0} follows from the compatibility property of Theorem
\ref{thm:LegSS}. This is the case when
$(M,\ker\alpha)=(J^1\Lambda,\ker\alpha_0)$ with the trivial
$\alpha_0$-Weinstein parametrization centered at the zero section, i.e. $\Phi(h)=j^1h$, since the contact
isotopy $(\phi_t)$, $\phi_t : (q,p,z)\mapsto (q,p-t\ud g(q),z-t g(q))$, is an
isotopy of strict contactomorphisms whose time $1$ satisfies
\eqref{eq:contstrict}. The $\alpha_0$-Hamiltonian of this isotopy is given by
the autonomous function $H : (q,p,z)\mapsto -g(q)$.

When $(M,\ker\alpha)$ is a general contact manifold and for $\delta$ small
enough, we get the desired result by cutting off the Hamiltonian function $H$
properly. More precisely let $\delta$ be small enough so that $(f_t:=f-tg)$,
$(g_t:=(1-t)g)$ and $(g_{t,s}^\alpha:=g_s+(2t-1)2\delta)\subset U$ for all
$f,g\in U_\delta$ and $t,s\in[0,1]$. Consider $\rho : J^1\Lambda\to\R$ a cutoff
function supported in $V$ such that $\rho$ is equal to $1$ on a neighborhood
containing $(j^1f_t)$, $(j^1g_t)$ and $( j^1g_{t,s}^\alpha)$. Let $K : M\to\R$
be the compactly supported $\alpha$-Hamiltonian function defined by
$x\mapsto\rho(\Psi(x)) H(\Psi(x))$ if $x\in \mathcal{V}$ and $x\mapsto 0$
otherwise. It is easy to check that its time $1$-flow $\phi$ commutes with
$\phi_\alpha^t$ on $\Phi(g)$ for $t\in [-2\delta,2\delta]$ and satisfies
\eqref{eq:contstrict} when $h$ is either $f$ or $g$. Moreover by the previous
Corollary \ref{cor:critval} and triangle inequality we get that
$\ell_\pm^\alpha(\Phi(f),\Phi(g))\in (-2\delta,2\delta)$. The conclusion
follows from the definition of $\ell_\pm^\alpha$.
\end{proof}

\begin{proof}[Proof of Theorem \ref{thm}] 
We will show Theorem \ref{thm} for $U'\to \Leg$ (resp. $U'\to\uLeg$) where $U':=U_{\delta_0/2}$ and $\delta_0$ is the positive constant of Lemma \ref{lem0}.

Remark that for any $f\in U_{\delta_0}$ convexity ensures that
    $tf\in U_{\delta_0}$ for all $t\in[0,1]$. Since the set of critical values
    $\mathrm{CV}(tf)=t\mathrm{CV}(f)$ of $tf$ is nowhere dense, we deduce by continuity of
    $\ell_+^\alpha(\cdot,\Lambda)$ (see (\ref{eq:dspecdSCH}) and below) and Corollary \ref{cor:critval} that there exists
    $x_0\in \Lambda$ a critical point of $f$ such that
    $\ell_+^\alpha(\Phi(tf),\Lambda)=tf(x_0)$. It thus remains to show that $f(x_0)=\max
    f$ for any $f\in U_{\delta_0/2}$ to get the desired equality for $\ell_+^\alpha$. 

    Let us first assume $f\geq 0$ and $f\in U_{\delta_0}$ is Morse.
    In particular $f\neq 0$.
    Let $\mathcal{B}$ be a Morse neighborhood of a maximum $x_1$ of $f$:
    there exist $\varepsilon>0$ and a diffeomorphism $Q:=(q_1,\ldots,q_n)$
    from $\mathcal{B}$ to the Euclidean ball of radius $\varepsilon$ centered at $0$
    such that $Q(x_1)=0$ and $f(x):=\max f
    -\sum q_i(x)^2$ for any $x\in\mathcal{B}$. It is then easy to
    construct a cut-off function $\rho : M\to [0,1]$ supported in $\mathcal{B}$
    such that $\rho(x_1)=1$ and $\mathrm{CV}(\rho f):=\{\max f,0\}$. 
    Note that $\rho f\leq f$ since $f\geq 0$ but $\rho f$ may not be contained
    in $U$. However there exists $\mu>0$ small enough so that $t\rho f\in U_\delta$ for all
    $t\in[0,\mu]$. Therefore thanks to Corollary \ref{cor:critval}
    $\ell_+^\alpha(\Phi(t\rho f),\Lambda)\in\{t\max f,0\}$. Moreover
    $\Phi(t\rho f)\ne\Lambda$ (resp. $\Pi(\Phi(t\rho f))\ne\Pi(\Lambda)$) for $t\in(0,\mu)$ since $f\ne 0$ therefore by
    non-degeneracy of the selectors $\ell_+^\alpha(\Phi(t\rho f),\Lambda)=t\max
    f$. Thus by monotonicity we deduce that
    $\ell_+^\alpha(\Phi(tf),\Lambda)=t\max f$.

    Let us now assume $f\in U_{\delta_0/2}$ is Morse but $-\delta_0/2<m:=\min f<0$.
    Consider then $g:=f-m\geq 0$. Remark that $\max g=\max
    f-m<\delta_0/2+\delta_0/2=\delta_0$. Therefore $g\in U_{\delta_0}$ is Morse and non-negative. So we deduce by the previous case that
    $\ell_+^\alpha(\Phi(g),\Lambda)=\max f-m$. Moreover
    $\phi_m^\alpha(\Phi(g))=\Phi(f)$ so by normalization we deduce that
    $\ell_+^\alpha(\Phi(f),\Lambda)=\max f$.

    Finally, let $f$ be any function in $U_{\delta_0/2}$. Since Morse functions are
    $C^2$-dense, one can find a sequence $(f_n)$ of Morse functions in $U_{\delta_0/2}$
    that converges to $f$ in the $C^2$-topology. Therefore $(\Phi(f_n))$
    $C^1$-converges to $\Phi(f)$. By $C^1$-continuity of $\ell_+^\alpha$ 
    (see (\ref{eq:dspecdSCH}) and below) and the
    previous cases, we deduce that $\ell_+^\alpha(\Phi(f),\Lambda)=\max f$.

    Thanks to Lemma \ref{lem0}
    $\ell_+^\alpha(\Phi(f),\Phi(g))=\ell_+(\Phi(f-g),\Lambda)=\max(f-g)$. 
    
    We deduce the analogous result for $\ell_-^\alpha$ by a similar reasoning or simply by using the Poincaré duality property, and this concludes the proof. 
\end{proof}

Let us conclude this section by deducing the $C^1$-local coincidence of
$\dspec^\alpha$ and $\dSCH^\alpha$.

\begin{proof}[Proof of Corollary~\ref{cor:shel}]
    As recalled at (\ref{eq:dspecdSCH}), one always
    has $\dspec^\alpha\leq\dSCH^\alpha$.
    In the neighborhood of $\Lambda = \Phi(0)$, one has
    $\dspec^\alpha(\Phi(f),\Phi(g)) = \max |f-g|$.
    The $\alpha_0$-Hamiltonian map associated with $(j^1 f_t)$ for
    $f_t := (1-t)f + tg$, $t\in [0,1]$, is $(q,p,z)\mapsto g(q)-f(q)$,
    therefore $\LSCH^{\alpha_0}(j^1 f_t) = \max |f-g|$.
    Since the SCH-length of an isotopy is invariant under strict contactomorphism
    $\LSCH^\alpha(\Phi(f_t)) = \max |f-g|$.
    As a consequence, one gets the
    reverse inequality $\dSCH^\alpha(\Phi(f),\Phi(g))\leq \max |f-g|$.
\end{proof}

\section{Characterization of the geodesics}\label{sec:preuvegeo}

To prove Theorem \ref{thm:geo} let us first state and prove the two following lemmata.

\begin{lem}\label{lem:int}
    Given a continuous map $g : [a,b]_t\times N\to\R$ on
    a compact set, the
    following two conditions are equivalent:
    \begin{enumerate}
        \item $\int_a^b \max | g_t| \ud t = \max_x\left| \int_a^b g_t(x)\ud t
            \right|$,
        \item there exist $\epsilon\in\{\pm1\}$ and $x_0\in \Lambda$ such that $\forall t\in [a,b]$,
            $\epsilon g_t(x_0) = \max | g_t|$.
    \end{enumerate}
\end{lem}

\begin{proof}
    The implication $(2)\Rightarrow (1)$ is clear.
    Conversely, let $x_0\in N$ and $\epsilon\in\{\pm 1\}$
    be such that $\epsilon\int_a^b g_t(x_0)\ud t =
    \max_x |\int_a^b g_t(x)\ud t|$.
    Then $t\mapsto \max | g_t| - \epsilon g_t(x_0)$ is a non-negative continuous
    map the integral of which vanishes over $[a,b]$ by assumption.
    The conclusion follows.
\end{proof}

\begin{lem}\label{lem:geo}
    Let $f:[0,1]_t\times N\to\R$ be a smooth map on a closed manifold $N$. Then
    $(j^1f_t)$ is a minimizing geodesic for
    $\ud^{\alpha_0}$ if and only if it is $\alpha_0$-quasi-autonomous, where
    $\alpha_0$ denotes the canonical $1$-form of $J^1N$ and $\ud^{\alpha_0}$
    either $\dspec^{\alpha_0}$ or $\dSCH^{\alpha_0}$.
\end{lem}

\begin{proof}
    It is enough to prove the result for $\dSCH^{\alpha_0}$ thanks to Corollary
    \ref{cor:shel} and the first point in Remark \ref{rem:longueur spectral et
    shelukhin}. Moreover thanks to Remark \ref{rem:longueur de shelukhin} $(j^1
    f_t)$ is a geodesic if and only if $\dSCH^{\alpha_0}(j^1 f_0,j^1 f_1) =
    \int_0^1 \max| H_t |\ud t$ where $(H_t)$ is the $\alpha_0$-Hamiltonian map
    of $(j^1 f_t)$.  But $H_t = \partial_t f_t\circ \pi|_{j^1 f_t}$ where
    $\pi:J^1 N\to N$ is the bundle map and $\dSCH^{\alpha_0}(j^1 f_0, j^1 f_1) =
    \max| f_1 - f_0 |$ thanks to Remark \ref{rem:1 jet}, so it boils down to 
    \begin{equation*}\label{eq:intft}
        \int_0^1 \max|\partial_t f_t| \ud t = \max_q\left| \int_0^1 \partial_t f_t(q)\ud
        t\right|.
    \end{equation*}
    By Lemma~\ref{lem:int}, this is equivalent to
    the existence of $q_0\in N$ and $\epsilon\in\{\pm 1\}$ such that $\forall t\in [0,1]$,
    $\epsilon \partial_t f_t(q_0) = \max | \partial_t f_t |$.

    Let assume that $(j^1 f_t)$ is a geodesic.
    In particular, $q_0$ is a critical point of $\partial_t f_t$ where $t\in [0,1]$
    is fixed. As the time-derivative commutes with the differential operator on $N$,
    $\partial_t(\ud f_t)$ vanishes at $q_0$ so $t\mapsto \ud f_t(q_0)$ is
    constant. Therefore $(j^1 f_t)$ is quasi-autonomous,
    by taking $x_t = (q_0,\ud f_0(q_0),f_t(q_0))$.

    Conversely, if $(j^1 f_t)$ is quasi-autonomous, the associated
    path of $x_t\in j^1 f_t$ belonging to a same Reeb orbit
    such that $\epsilon H_t(x_t) = \max| H_t|$, for some $\epsilon\in\{\pm
    1\}$, is necessarily of the form
    $(q_0,\ud f_0(q_0),f_t(q_0))$, $t\in [0,1]$.
    As $H_t = \partial_t f_t\circ\pi|_{j^1 f_t}$, the conclusion follows from
    the beginning of the proof.
\end{proof}

\begin{proof}[Proof of Theorem \ref{thm:geo}]
 It is enough to prove the result for $\dspec^\alpha$ thanks to Corollary \ref{cor:shel}.
 Let $(\Lambda_t)_{t\in I}$ be an isotopy of $\Leg$ (resp. $\uLeg$) and let
 us fix $t_0\in I$.
 Let us consider an $\alpha$-Weinstein parametrization $\Phi:U\to\Leg$
 (resp. $U\to\uLeg$)
 centered at $\Lambda_{t_0}$ and $U'\subset U$ given by Theorem~\ref{thm}.
 Let $J\subset I$ be a neighborhood of $t_0$ such that $(\Lambda_t)_{t\in J}
 \subset \Phi(U')$ and denote $f_t := \Phi^{-1}(\Lambda_t)$ for $t\in J$.
 According to Theorem~\ref{thm},
 \[
     \dspec^\alpha(\Lambda_t,\Lambda_s) = \dspec^{\alpha_0}(j^1 f_t, j^1 f_s),
     \quad \forall t,s \in J,
 \]
 (see also Remark~\ref{rem:1 jet}).
 Therefore, $(\Lambda_t)_{t\in J}$ is a minimizing geodesic for $\dspec^\alpha$
 if and only if $(j^1 f_t)$ is a minimizing geodesic for $\dspec^{\alpha_0}$
 in $J^1\Lambda_{t_0}$.
 The conclusion now follows from Lemma~\ref{lem:geo} as the notion of
 quasi-autonomy is stable under strict contactomorphisms.
\end{proof}

\section{The case of contactomorphisms}
\label{sec:contacto}

Let us deduce from the Legendrian case the analogous statements for
contactomorphisms when a stronger orderability condition is assumed.

\subsection{$C^1$-local flatness}

Let us first recall that
for any \textit{closed} cooriented contact manifold $(M,\xi:=\ker\alpha)$, the
$1$-form $\beta:=\pi_2^*\alpha-e^\theta\pi_1^*\alpha$ is a contact form on
$M\times M\times \R$ where $\pi_i : M\times M\times \R\to M$,
$(x_1,x_2,\theta)\mapsto x_i$. Its cooriented kernel $\Xi:=\ker\beta$ 
does not depend (up to isomorphism) on
the choice of the contact form $\alpha$ supporting $\xi$.
Note that the diagonal $\Delta:=\{(x,x,0)\ |\ x\in M\}$
is a closed Legendrian of $(M\times M\times\R,\Xi)$. For any contactomorphism
$\phi$ of $(M,\ker\alpha)$ isotopic to the identity, the graph of $\phi$,
$gr_\alpha(\phi):=\{(x,\phi(x),g(x))\ |\ x\in M\}$ where
$\phi^*\alpha=e^g\alpha$, lies in $\Leg(\Delta)$. Moreover the map
$\Gcont\to\Leg(\Delta)$ (resp. $\uGcont\to\uLeg(\Delta)$)
we have just described is a
local homeomorphism at the identity.

In particular this allows again to construct a continuous embedding
$U\to\Gcont$ (resp. $U\to\uGcont$) sending the zero map
to the identity where $U$ denotes a 
$C^2$-neighborhood of the zero map in $C^\infty(M,\R)$.
Such a map is called an $\alpha$-Weinstein parametrization centered
at the identity.

\begin{cor}\label{cor:ugcontlocflat}
    If $(M,\ker\alpha)$ is closed and $\uLeg(\Delta)$ is orderable then endowed with
    the spectral pseudo-norm $\uGcont$ is $C^1$-locally flat.
    More precisely, for every $\alpha$-Weinstein parametrization
    $\Phi:U\to\uGcont$ centered at the identity, there exists $U'\subset U$
    a $C^2$-neighborhood of the zero map
    such that for all $f\in U'$
    \[
        c_+^\alpha(\Phi(f))=\max f\quad \text{and}\quad c_-^\alpha(\Phi(f))=\min f\]
        in particular $\Nspec{\Phi(f)}^\alpha=\max |f|$. 
\end{cor}

\begin{proof}
    Let us define the map
    $\mathcal{C}_\pm^\alpha(\phi):=\ell_\pm(gr_\alpha(\phi),\Delta)$ where by a
    slight abuse of notation $\Delta\in\uLeg(\Delta)$ denotes the class of the
    constant path and $gr(\phi)\in\uLeg(\Delta)$ denotes the class of the path
    $(gr_\alpha(\phi_t))$ for a path $(\phi_t)\subset \Gcont$ representing
    $\phi\in\uGcont$. By Theorem \ref{thm} $\mathcal{C}_+^\alpha(\Phi(f))=\max
    f$, $\mathcal{C}_-^\alpha(\Phi(f))=\min f$.
    By maximality of
    $c_\pm^\alpha$ (see the discussion at the end of \cite[Section
    1.3]{allais2023spectral}) we deduce that $c_+^\alpha(\Phi(f)) \geq \max f$
    and $c_-^\alpha(\Phi(f)) \leq \min f$.
    The reverse inequalities come from (\ref{eq:GselectorsHam}).
\end{proof}

A proof similar to that of Corollary~\ref{cor:shel} then brings the corresponding
statement for $\uGcont$.

\begin{cor}\label{cor:shelukhin contact}
   Suppose that $(M,\ker\alpha)$ is closed and $\uLeg(\Delta)$ is orderable.
   Then for every $\phi\in\uGcont$ that is $C^1$-close to the identity, 
   \[\NSH{\phi}^\alpha=\Nspec{\phi}^\alpha.\]
   Therefore endowed with the Shelukhin-Hofer pseudo norm $\uGcont$ is $C^1$-locally flat. 
\end{cor}

    By mimicking the proof of Theorem \ref{thm} one can see that Corollary \ref{cor:ugcontlocflat} and Corollary \ref{cor:shelukhin contact}
    actually hold whenever there exist maps
$c_-\leq c_+$ that are spectral, compatible with the partial order, non-degenerate and normalized in the
sense of \cite{allais2023spectral}. In particular it holds for $\uGcont$ of
lens spaces \cite{allais2024spectral}. When $(M,\ker\alpha)$ is a closed contact manifold such that $\Gcont$ (resp.
$\uGcont$) is orderable, it is conjectured in \cite{allais2023spectral} that
$c_\pm^\alpha$ on $\Gcont$ (resp. $\uGcont$) are spectral. It is interesting to
note that for elements that can be joined to the identity by a minimizing
geodesic this conjecture holds. More precisely denote by 
\[
\mathcal{E}_-:=\left\{\phi\in \Gcont\ \text{(resp. }\in \uGcont\text{)}\ |\
\phi_{c_-^\alpha(\phi)}^\alpha\cleq\phi\right\} \quad \quad
\mathcal{E}_+:=\left\{\phi\in\Gcont\ \text{(resp. }\in \uGcont\text{)}\ |\
\phi\cleq\phi_{c_+^\alpha(\phi)}^\alpha\right\}.
\]

\begin{prop}\label{prop:spec}
    Let $\phi\in\mathcal{E}_\pm$ then
    $c_\pm^\alpha(\phi)\in\spec^\alpha(\phi)$.
\end{prop}

It seems
however unlikely that $\mathcal{E}_+=\mathcal{E}_-=\Gcont$ -- which is
equivalent to saying that $\{\phi\cgeq\id\}\subset\Gcont$ is closed for the
$C^1$-topology. See Section \ref{ssec:geocontact} for discussions. 

Given a pair of points $x,y$ in $\Gcont$ or $\uGcont$, let us write $x\cll y$ if there
exists a positive isotopy joining $x$ to $y$.
The statement of Proposition \ref{prop:spec} follows directly from the following lemma.

\begin{lem}\label{lem:spectralite}
    Let $\phi\in\Gcont$ (resp. $\uGcont$) such that $\id\cleq\phi$. If $\phi$ (resp. $\Pi(\phi)$) does not have any discriminant point then $\id\cll\phi$.
\end{lem}
Note that if $\Gcont$ (resp. $\uGcont$) is not orderable the proposition is trivial. 

\begin{proof}[Proof of Lemma \ref{lem:spectralite}] Consider
    $(\phi_t)\subset\Gcont$ a non-negative path starting at the identity such that $\phi_1=\phi\in\Gcont$ (resp. $[(\phi_t)]=\phi$).
    Since $\phi_1$ does not have discriminant point there exists
    $\varepsilon_0>0$ such that the path of closed Legendrian submanifolds
    $(\text{gr}(\phi_{-t\varepsilon}^\alpha\circ\phi_1))_{t\in[0,1]}$ does not
    intersect $\Delta$ for any $\varepsilon\in(0,\varepsilon_0)$. Therefore
    there exists a compactly supported contactomorphism $(\psi_{\varepsilon})$
    of $M\times M\times \mathbb{R}$ isotopic to the identity that sends
    $(\text{gr}(\phi_1))$ on $\text{gr}(\phi_{-\varepsilon}^\alpha\circ\phi_1)$ and
    that fixes $\Delta$. Moreover, for $\varepsilon$ sufficiently small, one
    can construct $\psi_\varepsilon$ sufficiently $C^1$-small such that the
    non-negative path of Legendrians $(\psi_{\varepsilon}(\text{gr}(\phi_t))$
    starting at $\Delta$ is graphical, \emph{i.e.} there exists an isotopy
    $(\varphi_t)\subset\Gcont$ starting at the identity such that
    $\psi_\varepsilon(\text{gr}(\phi_t))=\text{gr}(\varphi_t)$. This implies
    that $\id\cleq\varphi_1=\phi_{-\varepsilon}^\alpha\circ\phi_1\cll \phi_1$ and
    concludes the proof. 
\end{proof}

\begin{proof}[Proof of Proposition \ref{prop:spec}]
     Let $t:=c_-^\alpha(\phi)$. Since $\phi_t^\alpha\cleq\phi$ it implies that $\id\cleq\phi_{-t}^\alpha\phi$. Suppose by contradiction that $t$ is not in the spectrum. Therefore $\id\cll \phi_{-t}^\alpha\phi$ by Lemma \ref{lem:spectralite} which contradicts the definition of $t$. To deduce the result for $c_+^\alpha$ one can use Poincaré duality.
\end{proof}

\subsection{Geodesics}\label{ssec:geocontact}

Identifying contactomorphisms with their graphs in the contact product
of $M$, the definition of quasi-autonomous contact isotopies is straightforward.

\begin{definition}
A contact isotopy $(\phi_t)\subset\Gcont$ is $\alpha$-quasi-autonomous if the
corresponding Legendrian isotopy $(gr_\alpha(\phi_t))\subset (M\times M\times
\R,\ker\beta)$ is $\beta$-quasi-autonomous. 
\end{definition}

Note that a contact isotopy $(\phi_t)$ starting at the identity is
$\alpha$-quasi-autonomous if and only if there exist a point $x\in M$ and $\epsilon\in\{\pm 1\}$ such that
$x$ is an $\alpha$-translated point of $\phi_t$ 
(\emph{cf.} (\ref{eq:Gspectrum}) and above)
and $\epsilon H_t(\phi_t(x))=\max
|H_t|$ for all $t\in[0,1]$ where $H$ denotes the $\alpha$-Hamiltonian function
of $(\phi_t)$. 

The characterization of geodesics in this context is now a direct consequence
of Theorem~\ref{thm:geo} and the $C^1$-local isometry between $\Gcont$ or $\uGcont$
and $\uLeg(\Delta)$.
Let us recall that a geodesic for the group pseudo-norm $|\cdot|$ is by definition
a geodesic for the right-invariant pseudo-distance
$(g,h)\mapsto |g h^{-1}|$ (in fact it will also be a geodesic for the
associated left-invariant pseudo-distance in our case).

\begin{cor}\label{cor:geo}
Let $(M,\ker\alpha)$ be a closed contact manifold such that $\uLeg(\Delta)$ is
orderable. A contact isotopy $(\phi_t)$ is a geodesic for $|\cdot|^\alpha$ if
and only if it is $\alpha$-quasi-autonomous, where $|\cdot|^\alpha$ denotes
either $\Nspec{\cdot}^\alpha$ or $\NSH{\cdot}^\alpha$.
\end{cor}

The second author in \cite{Arlove2023} characterized some minimizing geodesics of the Shelukhin-Hofer norm on the identity component of the group of compactly supported contactomorphisms of $\R^{2n}\times \sphere{1}$ endowed with its standard contact form $\alpha_{st}$ 
. Since $\R^{2n}\times \sphere{1}$ is non compact, the previous Corollary \ref{cor:geo} does not cover this case. However the geodesics characterized in \cite{Arlove2023} are indeed special cases of $\alpha_{st}$-quasi-autonomous isotopies. It would be interesting to extend the selectors $\ell_\pm^\alpha$ and $c_\pm^\alpha$ to compactly supported isotopies and extend the results of this paper to the non compact settings (see also \cite[Remark 1.4]{allais2023spectral}).   

\begin{questions}\label{question}\
\begin{enumerate}
    \item Does it exist $\Lambda_1\in\Leg(\Lambda_0)$ or $\uLeg$ (resp.
        $\phi_1\in\Gcont$ or $\uGcont$) that cannot be attained by minimizing
        smooth geodesics, \emph{i.e.} for any isotopy $(\Lambda_t)\subset\Leg$ or
        $\uLeg$ (resp. $(\phi_t)\subset\Gcont$ or $\uGcont$)
    \[
    \LSCH^\alpha(\Lambda_t) =
    \mathrm{Length}_{\dspec^\alpha}(\Lambda_t) > \dSCH^\alpha(\Lambda_1,\Lambda_0) \geq
    \dspec^\alpha(\Lambda_1,\Lambda_0)
    \]
    \[\text{ (resp. } 
        \LSH^\alpha(\phi_t)=
        \mathrm{Length}_{\Nspec{\cdot}^\alpha}(\phi_t)>
        \NSH{\phi_1}^\alpha\geq \Nspec{\phi_1}^\alpha)\text{?}
    \]
    \item Does Shelukhin-Chekanov-Hofer type distance (resp. norm) agree with the spectral distance (resp. spectral norm)?
    \item Does orderable $\Leg$ (resp. $\Gcont$) endowed with the spectral
        distance (resp. spectral norm) is an intrinsic metric space (when
        taking the infimum of length over continuous paths for the topology
        induced by the distance, \emph{i.e}. the interval topology
        \cite{CheNem2020,Nakamura2023})?
\end{enumerate}
\end{questions}
Note that a negative answer to question (3) implies a negative answer to
question (2). Moreover a negative answer to question (2) implies a positive answer to
question (1) for the spectral type distances.

 To answer positively to question (1) one can try to adapt some construction of
 Lalonde-McDuff. Indeed in the Hamiltonian case
 Lalonde-McDuff \cite[Prop 5.1 Part I]{lalondemcduff12} constructed examples of
 $\tpsi\in\widetilde{\ham}(\CP^1)$ that cannot be joined to the identity by any
 minimizing geodesics $\{\psi_t\}\subset\ham(\CP^1)$ for the Hofer length. It
 would be interesting to investigate whether certain lifts
 $\tphi\in\uGcont(\RP^3)$ of $\tpsi$ satisfy a similar property : they cannot
 be joined to the identity by minimizing geodesics of the Hofer-Shelukhin
 length. Moreover since Corollary \ref{cor:shelukhin contact} and the discussion following it imply that the
 Hofer-Shelukhin length and the spectral length agree on smooth
 paths in this context, such $\tphi$ would be examples of elements lying inside
$\uGcont\setminus\mathcal{E}_\pm$. 

To answer negatively to question (2) it would be enough to show that the
Shelukhin-Hofer type distance (resp. norm) is not compatible with the partial
order.

\bibliographystyle{amsplain}
\bibliography{biblio} 

\providecommand{\bysame}{\leavevmode\hbox to3em{\hrulefill}\thinspace}
\providecommand{\MR}{\relax\ifhmode\unskip\space\fi MR }
\providecommand{\MRhref}[2]{%
  \href{http://www.ams.org/mathscinet-getitem?mr=#1}{#2}
}
\providecommand{\href}[2]{#2}
\begin{thebibliography}{10}

\bibitem{allais2023spectral}
Simon {Allais} and Pierre-Alexandre {Arlove}, \emph{{Spectral selectors and contact orderability}}, arXiv e-prints (2023), arXiv:2309.10578.

\bibitem{allais2024spectral}
Simon {Allais}, Pierre-Alexandre {Arlove}, and Sheila {Sandon}, \emph{{Spectral selectors on lens spaces and applications to the geometry of the group of contactomorphisms}}, arXiv e-prints (2024), arXiv:2402.13689.

\bibitem{ArlovePhD}
Pierre-Alexandre Arlove, \emph{{Normes sur le groupe des contactomorphismes et contactisation de domaines {\'e}toil{\'e}s, leurs g{\'e}od{\'e}siques et leurs caract{\'e}ristiques translat{\'e}es}}, Phd thesis, {Universit{\'e} de Strasbourg}, July 2021.

\bibitem{Arlove2023}
\bysame, \emph{Geodesics of norms on the contactomorphisms group of {$\Bbb R^{2n}\times S^1$}}, Journal of Fixed Point Theory and Applications \textbf{25} (2023), no.~4, 80.

\bibitem{bhupal}
Mohan Bhupal, \emph{A partial order on the group of contactomorphisms of {$\Bbb R^{2n+1}$} via generating functions}, Turkish J. Math. \textbf{25} (2001), no.~1, 125--135.

\bibitem{BiaPol1994}
Misha Bialy and Leonid Polterovich, \emph{Geodesics of {H}ofer's metric on the group of {H}amiltonian diffeomorphisms}, Duke Math. J. \textbf{76} (1994), no.~1, 273--292.

\bibitem{chernov2011legendrian}
Vladimir Chernov and Stefan Nemirovski, \emph{Legendrian links, causality, and the low conjecture}, Journal of Differential Geometry \textbf{88} (2011), no.~3, 425--484.

\bibitem{chernovnemirovski2}
\bysame, \emph{Universal orderability of {L}egendrian isotopy classes}, J. Symplectic Geom. \textbf{14} (2016), no.~1, 149--170.

\bibitem{CheNem2020}
\bysame, \emph{Interval topology in contact geometry}, Commun. Contemp. Math. \textbf{22} (2020), no.~5, 1950042, 19.

\bibitem{CFP}
Vincent Colin, Emmanuel Ferrand, and Petya Pushkar, \emph{Positive isotopies of {L}egendrian submanifolds and applications}, Int. Math. Res. Not. IMRN (2017), no.~20, 6231--6254.

\bibitem{EKP}
Yakov Eliashberg, Sang~Seon Kim, and Leonid Polterovich, \emph{Geometry of contact transformations and domains: orderability versus squeezing}, Geom. Topol. \textbf{10} (2006), 1635--1747.

\bibitem{EP00}
Yakov Eliashberg and Leonid Polterovich, \emph{Partially ordered groups and geometry of contact transformations}, Geom. Funct. Anal. \textbf{10} (2000), no.~6, 1448--1476.

\bibitem{Gro85}
Mikhael Gromov, \emph{Pseudo holomorphic curves in symplectic manifolds}, Invent. Math. \textbf{82} (1985), no.~2, 307--347.

\bibitem{gromovmetrique}
Mikhail Gromov, \emph{Metric structures for riemannian and non-riemannian spaces}, 01 2007.

\bibitem{Hedicke2022}
Jakob Hedicke, \emph{Lorentzian distance functions in contact geometry}, J. Topol. Anal. (2022), 1--21.

\bibitem{lalondemcduff12}
Fran\c{c}ois Lalonde and Dusa McDuff, \emph{Hofer's {$L^\infty$}-geometry: energy and stability of {H}amiltonian flows. {I}, {II}}, Invent. Math. \textbf{122} (1995), no.~1, 1--33, 35--69.

\bibitem{geometricvariants}
Dusa McDuff, \emph{Geometric variants of the {H}ofer norm}, J. Symplectic Geom. \textbf{1} (2002), no.~2, 197--252.

\bibitem{mcduffsalamon}
Dusa McDuff and Dietmar Salamon, \emph{Introduction to symplectic topology}, third ed., Oxford Graduate Texts in Mathematics, Oxford University Press, Oxford, 2017.

\bibitem{Nakamura2023}
Lukas {Nakamura}, \emph{{A new metric on the contactomorphism group of orderable contact manifolds}}, arXiv e-prints (2023), arXiv:2307.10905.

\bibitem{Nemirovski_dejavu}
Stefan Nemirovski, \emph{Legendrian links and déjà vu moments}, Journal of Geometry and Physics \textbf{192} (2023), 104950.

\bibitem{Ohnorm}
Yong-Geun Oh, \emph{Chain level {F}loer theory and {H}ofer's geometry of the {H}amiltonian diffeomorphism group}, Asian J. Math. \textbf{6} (2002), no.~4, 579--624.

\bibitem{Ohplat}
\bysame, \emph{Spectral invariants, analysis of the {F}loer moduli space, and geometry of the {H}amiltonian diffeomorphism group}, Duke Math. J. \textbf{130} (2005), no.~2, 199--295.

\bibitem{polterovich}
Leonid Polterovich, \emph{The geometry of the group of symplectic diffeomorphisms}, Lectures in Mathematics ETH Z\"{u}rich, Birkh\"{a}user Verlag, Basel, 2001.

\bibitem{RosenZhang2020}
Daniel Rosen and Jun Zhang, \emph{Chekanov's dichotomy in contact topology}, Math. Res. Lett. \textbf{27} (2020), no.~4, 1165--1193.

\bibitem{Sandon2010}
Sheila Sandon, \emph{An integer-valued bi-invariant metric on the group of contactomorphisms of {$\Bbb R^{2n}\times S^1$}}, J. Topol. Anal. \textbf{2} (2010), no.~3, 327--339.

\bibitem{schwarz}
Matthias Schwarz, \emph{On the action spectrum for closed symplectically aspherical manifolds}, Pacific J. Math. \textbf{193} (2000), no.~2, 419--461.

\bibitem{shelukhin}
Egor Shelukhin, \emph{The {H}ofer norm of a contactomorphism}, J. Symplectic Geom. \textbf{15} (2017), no.~4, 1173--1208.

\bibitem{tsuboi3}
Takashi Tsuboi, \emph{On the simplicity of the group of contactomorphisms}, Groups of diffeomorphisms, Adv. Stud. Pure Math., vol.~52, Math. Soc. Japan, Tokyo, 2008, pp.~491--504.

\bibitem{Vit}
Claude Viterbo, \emph{Symplectic topology as the geometry of generating functions}, Math. Ann. \textbf{292} (1992), no.~4, 685--710.

\end{thebibliography}
\end{document}